\documentclass{daj}
\usepackage{amsmath}
\usepackage{latexsym,amssymb,amsthm}

\newcommand\A{\mathrm{A}}

\renewcommand\S{\mathrm{S}}
\newcommand\Aut{\mathrm{Aut}}
\renewcommand\P{\mathcal{P}}
\newcommand\fix{\mathrm{fix}}
\newcommand\supp{\mathrm{supp}}
\newcommand\M{\mathcal{M}}

\newcommand\Q{\mathcal{Q}}
\newcommand\aug{\fboxsep=-\fboxrule\!\!\!\fbox{\strut}\!\!\!}
\newcommand\Z{\mathbb{Z}}

\theoremstyle{plain}
\newtheorem{theorem}{Theorem}[section]
\newtheorem{lemma}[theorem]{Lemma}
\newtheorem{corollary}[theorem]{Corollary}
\newtheorem{proposition}[theorem]{Proposition}

\theoremstyle{definition}

\newtheorem{remark}[theorem]{Remark}

\dajAUTHORdetails{%
  title = {Regular bipartite multigraphs have many (but not too many) symmetries}, 
  author = {Peter J. Cameron, Coen del Valle, and Colva M. Roney-Dougal},
  plaintextauthor = {Peter J. Cameron, Coen del Valle, Colva M. Roney-Dougal},
  runningtitle = {Regular bipartite multigraphs have many (but not too many) symmetries}, 
  keywords = {graph automorphisms, bipartite graphs, multigraphs, random matrices, set partitions},
}   

\dajEDITORdetails{%
   year={2025},
   number={22},
   received={31 May 2024},   
   published={6 October 2025},  
   doi={10.19086/da.144894},       
}   

\begin{document}

\begin{frontmatter}[classification=text]

\title{Regular bipartite multigraphs have \\ many (but not too many) symmetries} 

\author[Cam]{Peter J. Cameron}
\author[del]{Coen del Valle\thanks{Supported by the Natural Sciences and Engineering Research Council of Canada (NSERC), [funding reference number PGSD-577816-2023], as well as a University of St Andrews School of Mathematics and Statistics Scholarship.}}
\author[Ron]{Colva M. Roney-Dougal}

\begin{abstract}
Let $k$ and $l$ be integers, both at least 2. A \emph{$(k,l)$-bipartite graph} is an $l$-regular bipartite multigraph with coloured bipartite sets of size $k$. Define $\chi(k,l)$ and $\mu(k,l)$ to be the minimum and maximum order of automorphism groups of $(k,l)$-bipartite graphs, respectively. 
We determine $\chi(k,l)$ and $\mu(k,l)$ for $k\geq 8$, and analyse the generic situation when $k$ is fixed and $l$ is large. In particular, we show that almost all such graphs have automorphism groups which fix the vertices pointwise and have order far less than $\mu(k,l)$. These graphs are intimately connected with both contingency tables with uniform margins and uniform set partitions; 
we examine the uniform distribution on the set of $k\times k$ contingency tables with uniform margin~$l$, showing that with high probability all entries stray far from the mean. We also show that the symmetric group acting on uniform set partitions is non-synchronizing.
\end{abstract}
\end{frontmatter}

\section{Introduction}
Frucht's Theorem~\cite{frucht} is a fundamental result bridging the gap between group theory and graph theory. It states that every finite abstract group is the automorphism group of some finite graph. In contrast to Frucht's Theorem, the Erd\H{o}s--R\'enyi Theorem~\cite{erren} tells us that almost all finite simple graphs have trivial automorphism group.

Both of these results have seen far-reaching extensions. In the case of Frucht's
Theorem, there are many classes of combinatorial structures, including
regular and strongly regular graphs, Latin squares, and Steiner triple systems, 
where similar universality results hold (see e.g.~\cite{men,phelps,sab}). Kantor defined in~\cite{kantor} a family of finite permutation groups to be
\emph{universal} if every finite group is isomorphic to a 2-point stabiliser in
some group in the family. This generalises Frucht's theorem in the following way: consider the set $\Omega$ of all graphs with vertex set $[n]:=\{1,2,\ldots,n\}$. Each graph $\Gamma\in \Omega$ can be naturally identified with an element $\gamma$ of the elementary abelian group $2^{n\choose 2}$, so $2^{n\choose 2}$ acts on $\Omega$ via the symmetric difference of edge sets. Moreover, $\S_n$ acts on $\Omega$ by vertex permutation, giving an action of $2^{n\choose 2} :  \S_n$ on $\Omega$. 
The 2-point stabilisers are $$\S_n\cap(\gamma\S_n\gamma)=\{\sigma\in\S_n : \gamma\sigma\gamma\in \S_n\}=\Big\{\sigma\in\S_n : \sigma^{-1}\gamma\sigma\gamma\in \S_n\cap2^{n\choose 2} =1\Big\}=
\{\sigma\in\S_n : \Gamma^\sigma = \Gamma\}=\Aut(\Gamma),$$ for some graph $\Gamma$. 
Frucht's Theorem says that this family of permutation groups is universal.

In the case of the Erd\H{o}s--R\'enyi Theorem, for certain types of structure such as Latin squares and Steiner
triple systems it is known that almost all instances have trivial automorphism groups (see~\cite{mw} and~\cite{babai}, respectively).
But there are other structures for which this is not the case: for example almost all finite trees have a non-trivial automorphism~\cite{erren}. Moreover,
the automorphism group of every finite tree is built from the trivial group by
the operations of direct product and of wreath product with a symmetric group~\cite{jordan},
so there exist groups which are not the automorphism group of any tree.

A \emph{multigraph} is a triple $\Gamma=(V,E,I)$ where $V$ is a set of \emph{vertices}, $E$ is a set of \emph{edges}, and $I$ is an incidence relation such that for each $e\in E$ there exist exactly two distinct vertices $u$ and $v$ incident with $e$; we typically omit `$I$' and write $\Gamma=(V,E)$. Note that if we also require that any pair of vertices are mutually incident with at most one edge then we recover the definition of a simple graph. An \emph{automorphism} of the vertex-coloured multigraph $\Gamma=(V,E)$ is a pair $(\varphi_V,\varphi_E)$ where $\varphi_V$ is a colour preserving bijection from $V$ to itself and $\varphi_E$ is a bijection from $E$ to itself such that $(v,e)\in I$ if and only if $(\varphi_V(v),\varphi_E(e))\in I$. We use $\Aut(\Gamma)$ to denote the group of automorphisms of $\Gamma$. Note that this is different from the notion of automorphism used sometimes to refer simply to the vertex maps.

Given integers $k$ and $l$, a $(k,l)$-\emph{bipartite graph} is an $l$-regular bipartite multigraph with bipartite blocks of size $k$, one of which is coloured black, and the other white. Define $$\mathcal{A}(k,l):=\{|\Aut(\Gamma)| : \Gamma \text{ a $(k,l)$-bipartite graph}\},$$ and set \begin{equation}\label{chi}\tag{$*$}\chi(k,l):=\min(\mathcal{A}(k,l))\hspace{.5cm}\text{ and }\hspace{.5cm}\mu(k,l):=\max(\mathcal{A}(k,l)).\end{equation} 

We are interested in the symmetric
group of degree $kl$ acting on \emph{$(k,l)$-partitions}: partitions of $[kl]$ into $k$ parts of size $l$. The stabiliser of one point is the wreath product
$\S_l\wr \S_k$, and we will see that the 2-point stabilisers are automorphism
groups of $(k,l)$-bipartite graphs. For small $l$, the analogues of the Frucht and
Erd\H{o}s--R\'enyi theorems hold. (For the first, this is a result of
James~\cite{james2}, showing that these groups are universal in Kantor's sense.) But for
large $l$, there are multiple edges, so the automorphism group has a normal
subgroup which fixes all vertices and acts as the symmetric group on the set
of edges joining each pair of vertices. We are mostly concerned with the orders of such groups: our first main result is the following.

\begin{theorem}\label{mainmin}
Let $k\geq 8$, $l\geq2$ be integers and write $l=qk+r$ with $-2\leq r\leq k-3$. Then $\mu(k,l)=k!l!^k$, and if $l=2$ then $\chi(k,l)=2k$, otherwise
$$\chi(k,l)=\begin{cases}
(q+1)!^kq!^{k^2-2k}(q-1)!^k&\text{ if $r=0$},\\
(q+1)!^{rk}q!^{k^2-rk}&\text{ if $3\leq r\leq k-3$},\\
(q+r)!^{k+\lceil k/2\rceil }q!^{k^2-k-2\lceil k/2\rceil}(q-r)!^{\lceil k/2\rceil}&\text{ if $r=\pm1$ },\\
(q+\varepsilon)!^{2k+1}q!^{k^2-2k-2}(q-\varepsilon)!&\text{ if $r=2\varepsilon$ with $\varepsilon=\pm1$}.\\
\end{cases}$$
\end{theorem} 
The maximum is attained by the $(k,l)$-bipartite graph consisting of $k$ black-white pairs, with exactly $l$ edges between the vertices within each pair, and no edges crossing between pairs. Our next main result states that almost all $(k,l)$-bipartite graphs have far fewer automorphisms than the maximum.
\begin{theorem}\label{mainrand}
Fix an integer $k\geq 2$ and $\epsilon>0$. Let $l$ be an integer, and let $\Gamma$ be a uniformly chosen $(k,l)$-bipartite graph. Then asymptotically almost surely \begin{itemize}\item[(i)] if $k\geq 3$ then every vertex of $\Gamma$ is fixed by $\Aut(\Gamma)$; and 
\item[(ii)] $|\Aut(\Gamma)|\leq l!^k/l^{{(k-1)}^2-\epsilon},$
\end{itemize}
as $l\to \infty$.
\end{theorem}
\begin{remark}
Statement (ii) still holds in the case where we are instead choosing uncoloured bipartite graphs. This follows from the proof by noting that the total number of vertex permutations of an uncoloured $l$-regular bipartite graph with parts of size $k$ is bounded above by the constant $2k!$.
\end{remark}

In Section~\ref{prelim} we introduce a class of matrices which we call \emph{$(k,l)$-intersection matrices}, also known in the literature as integer doubly-stochastic matrices or uniform contingency tables, among other names. These are the bipartite adjacency matrices of $(k,l)$-bipartite graphs, that is, $k\times k$ matrices of non-negative integers with all row and column sums equal to $l$. Much of our study of $(k,l)$-bipartite graphs is via these $(k,l)$-intersection matrices. We prove an additional result about them which may be of independent interest.
\begin{theorem}\label{main4}
Let $k\geq2$ and $l$ be positive integers and let $f(l)$ be any $o(l)$ function. Let $X$ be a uniformly chosen $(k,l)$-intersection matrix. Then asymptotically almost surely $\min(|X_{i,j}-l/k|)>f(l)$ as $l\to \infty$.
\end{theorem}
In terms of $(k,l)$-bipartite graphs, Theorem~\ref{main4} implies that every black-white pair of vertices of a uniformly chosen $(k,l)$-bipartite graph is expected to have edge multiplicity straying far from the mean $l/k$ when $l$ is large (see Corollary~\ref{edges}). This suggests why one should not expect a randomly chosen $(k,l)$-bipartite graph to have few symmetries.

In addition to the family of symmetric groups acting on uniform partitions being universal, we show that these groups are also \emph{non-synchronizing}. A
permutation group $G$ acting on a set $\Omega$ is \emph{synchronizing} if for
every map $f:\Omega\to\Omega$ which is not a permutation, the monoid
$\langle G,f\rangle$ contains an element sending all of $\Omega$ to a single point. Synchronizing groups are known to be primitive, but classifying which families of primitive groups are synchronizing remains an open problem. Our final main result is a contribution to this classification.

\begin{theorem}\label{syncthm}
Let $k$ and $l$ be positive integers with $k\geq 3$ and $l\geq2$. Then both of the symmetric and alternating groups of degree $kl$ acting on the set of all $(k,l)$-partitions are
non-synchronizing.
\end{theorem}

The structure of the paper is as follows. In Section~\ref{prelim} we cover some of the preliminary structural concepts and tools. Section~\ref{minaut} is dedicated to proving Theorem~\ref{mainmin}. In Section~\ref{probaut} we carry out a probabilistic investigation of the problem, proving Theorems~\ref{mainrand} and~\ref{main4}. Finally, we wrap up the paper with a brief note on synchronization in Section~\ref{sync}.

\section{Preliminaries}\label{prelim}
In this section we will describe a natural way of associating a $(k,l)$-bipartite graph with each pair of $(k,l)$-partitions. Given positive integers $k$ and $l$ define $\S_{k\times l}$ to be the symmetric group of degree $kl$ acting on the set of all $(k,l)$-partitions. 

We first describe an intermediate object linking $(k,l)$-partitions and $(k,l)$-bipartite graphs. 
Denote by $\mathcal{M}_{a,b}$ the set of $a\times b$ matrices with entries in $\mathbb{Z}_{\geq 0}$.
Given $(k,l)$-partitions $\mathcal{P}=\{P_1,\dots, P_k\}$ and $\mathcal{Q}=\{Q_1,\dots, Q_k\}$ the \emph{intersection matrix} $M=M(\mathcal{P},\mathcal{Q})$ is the matrix $(|P_i\cap Q_j|)_{ij}\in\mathcal{M}_{k,k}$; we will usually order our partitions by the smallest element in each part so that $M$ is well-defined. We say that $M$ is a $(k,l)$-\emph{intersection matrix}. 
Since $\mathcal{P}$ and $\mathcal{Q}$ are partitions, $$\sum_{j=1}^k m_{ij}=\sum_{j=1}^k|P_{i}\cap Q_{j}|=|P_{i}|=l$$ for each $i$, and by symmetry the same is true for columns. In fact, the converse is true.

\begin{lemma}\label{mat}
Let $N=(n_{ij})\in\mathcal{M}_{k,k}$ and let $\mathcal{P}$ be a $(k,l)$-partition for some $l$. Then there is a canonical $(k,l)$-partition $\mathcal{Q}=\mathcal{Q}(N,\mathcal{P})$ with $N=M(\mathcal{P},\mathcal{Q})$ if and only if all row and column sums of $N$ are equal to $l$.
\end{lemma}
\begin{proof}
Set $Q_{1}=\bigcup_{i=1}^k R_{i1}$ where $R_{i1}$ is the first $n_{i1}$ elements of $P_i$. Given $Q_{1},Q_{2},\dots,Q_{j}$, define $Q_{j+1}=\bigcup_{i=1}^k R_{i(j+1)}$ where $R_{i(j+1)}$ is the first $n_{i(j+1)}$ elements of $P_{i}\setminus\left( \bigcup_{m=1}^j Q_{j}\right)$. This is possible since $$\left|P_{i}\setminus\left( \bigcup_{m=1}^j Q_{j}\right)\right|=l-\sum_{s\leq j} n_{is}=\sum_{s>j} n_{is},$$ for all $i$. Let $\mathcal{Q}:=\{Q_{1},\dots, Q_{k}\}$. By construction, the sets in $\mathcal{Q}$ are pairwise disjoint, and $|Q_{i}|=\sum_{s=1}^k n_{si}=l$ for all $i$, so $\mathcal{Q}$ is indeed a $(k,l)$-partition. Moreover, $|P_{i}\cap Q_{j}|=|P_{i}\cap R_{ij}|=n_{ij}$ for all $i,j$, so the result holds.
\end{proof}

We now show how to determine the stabiliser up to $\S_{kl}$-conjugacy of a pair of $(k,l)$-partitions from their intersection matrix. Let $a,b\in\mathbb{N}$. Then $\S_a\times \S_b$ acts on $\mathcal{M}_{a,b}$ via $$N^{(\sigma_1,\sigma_2)}=(n_{ij})^{(\sigma_1,\sigma_2)}=(n_{i^{\sigma_1} j^{\sigma_2}}).$$
Let $\mathcal{P}$ be a $(k,l)$-partition, and $H\leq (\S_{k\times l})_\mathcal{P}$. We say that $\mathcal{P}$ is \emph{partwise fixed} by $H$ if the induced action of $H$ on the parts of $\mathcal{P}$ is trivial.

\begin{lemma}\label{top}
Let $(\mathcal{P}_1,\mathcal{P}_2)$ be a pair of $(k,l)$-partitions with $N=M(\P_1,\P_2)$, let $K=\S_{k}\times \S_{k}$, and set $G=\S_{k\times l}$. 
Then the pointwise stabiliser $H:=G_{\mathcal{P}_1,\mathcal{P}_2}$ is permutation isomorphic to $$\left(\prod_{i,j\in [k]}\mathrm{Sym}(P_{1i}\cap P_{2j})\right) :  K_{N}.$$ In particular, $|H|=|K_{N}|\prod_{i,j}n_{ij}!$, and $H\cong G_{\Q_1,\Q_2}$ whenever $M(\Q_1,\Q_2)=N$. Moreover, $\P_1$ and $\P_2$ are both partwise fixed by $H$ if and only if $K_N=1$.
\end{lemma}
\begin{proof}
The kernel of the intransitive representation $\rho:H\to \mathrm{Sym}(\P_1)\times \mathrm{Sym}(\P_2)$ is  $$\ker\rho=\prod_{i,j\leq k}\mathrm{Sym}(P_{1i}\cap P_{2j});$$ fixing an ordering on the points of $[kl]$ yields a complement to $\ker \rho$ isomorphic to $K_N$.
\end{proof}

From here on we use $K$ to denote the group $\S_a\times \S_b$ when $a$ and $b$ are clear from context, and call an $a\times b$ matrix $N$ \emph{partwise fixed} if $K_N=1$. By Lemma~\ref{top}, $M(\P_1,\P_2)$ is partwise fixed if and only if $\P_1$ and $\P_2$ are partwise fixed by $(\S_{k\times l})_{\P_1,\P_2}$.

We now connect $(k,l)$-partitions and $(k,l)$-bipartite graphs. For any positive integers $a$ and $b$ we can view $N=(n_{ij})\in\M_{a,b}$ as a bipartite adjacency matrix with rows and columns indexed by black and white vertices, respectively, where $n_{ij}$ is the multiplicity of the edge between the $i$th black vertex and $j$th white vertex. Let $\Gamma$ be a graph with adjacency matrix $N$. Then $\Aut(\Gamma)=\Aut_{E}(\Gamma) :  \Aut_V(\Gamma)$ where $\Aut_{E}(\Gamma)$ is all automorphisms of $\Gamma$ which fix all vertices, and $\Aut_V(\Gamma)$ is a complement giving the induced action on vertices. If $N$ is a $(k,l)$-intersection matrix, then $\Gamma$ is a $(k,l)$-bipartite graph, $\Aut_{E}(\Gamma)\cong \prod \S_{n_{ij}}$, and $\Aut_{V}(\Gamma)\cong K_N$, whence \begin{equation}\label{autg}\tag{$**$}\Aut(\Gamma)=\Aut_E(\Gamma) :  \Aut_V(\Gamma)\cong \left(\prod\S_{n_{ij}}\right) :  K_N.\end{equation} The following lemma is therefore immediate. Recall first the definition~\eqref{chi} of $\chi(k,l)$.
\begin{lemma}\label{minpq}
Let $k$ and $l$ be integers each at least 2. Then $$\chi(k,l)=\min\{|(\S_{k\times l})_{\P,\Q}| : \P,\Q \text{ $(k,l)$-partitions}\}.$$
\end{lemma}
We now obtain some estimates of products of factorials. 
Let $s$ and $t$ be positive integers. Consider a sequence of \emph{multiplicities} $m_0,m_1,m_2,\dots,m_d$ with $m_0=0$ and $\sum m_i<t$, and a sequence of \emph{ranks} of the form $$r_1>r_2>\cdots> r_d>\lceil s/t\rceil\quad\text{ or }\quad r_1<r_2<\cdots< r_d<\lfloor s/t \rfloor,$$ such that $\sum m_ir_i<s$.  Let $X=[(m_1,r_1),(m_2,r_2),\dots, (m_d,r_d)]$ 
 --- we call $X$ a \emph{multiplicity sequence}. If $A_0$ is the multiset with $m_i$ entries equal to $r_i$ for each $i\geq 1$, then $A_0$ can be completed to a multiset $A$ of size $t$ with sum $s$ by adding only elements equal to either the floor or ceiling of the average required value, which is $$x=x(s,t,X):=\frac{s-\left(\sum m_ir_i\right)}{t-\sum m_i}.$$ 
 Let $b=b(s,t,X)\leq t-\sum m_{i}$ be a positive integer satisfying $$s-\sum m_ir_i=b\lceil x\rceil +\left(t-\left(\sum m_i\right)-b\right)\lfloor x\rfloor,$$ and set $S=S(s,t,X)$ to be the collection of all non-negative integer multisets $A$ of size $t$ with sum $s$ such that for every $1\leq i\leq d$ the multiset $A$ has at least $\sum_{j=1}^i m_j$ entries which are at least $r_i$ if $r_i> \lceil s/t\rceil$ or at most $r_i$ if $r_i<\lfloor s/t\rfloor$.

\begin{lemma}\label{minN}
Let $s$ and $t$ be positive integers and let $X$ be a multiplicity sequence. Set $x=x(s,t,X)$, $b=b(s,t,X)$, and $S=S(s,t,X)$. Then $\min\left\{\prod_{a\in A} a! : A\in S\right\}$ is achieved uniquely by the integer multiset $B\in S$ containing each rank with its associated multiplicity, and all other elements as close to each other as possible. Thus, this minimum is precisely $$\left(\prod_{i}r_i!^{m_i}\right)\lceil x\rceil!^b\lfloor x\rfloor!^{t-\left(\sum m_i\right)-b}.$$ 
\end{lemma}
\begin{proof}
We prove the result for $r_d>\lceil s/t\rceil$; the other case is similar. Let $A\in S$ be such that $\prod_{a\in A} a!$ is minimum, and suppose $A\ne B$. 
Order the elements of $A=\{a_1,a_2,\dots, a_t\}$ and $B=\{b_1,b_2,\dots, b_t\}$ in non-increasing order so that $$a_1,b_1,a_2,b_2,\dots ,a_{\sum_{j\leq i} m_j},b_{\sum_{j\leq i} m_j}\geq r_i\text{ for all $i$.}$$ Since $\sum a_i=\sum b_i=s$ and both $A$ and $B$ are indexed in non-increasing order, there exist indices $j<h$ such that $a_j>b_j$ and $a_h<b_h$. In particular $a_j>a_h+1$. Let $A'$ be the multiset obtained from $A$ by replacing $a_j$ with $a_j-1$, and $a_h$ with $a_h+1$. Then $A'\in S$ and $$\prod_{a\in A'} a!=\frac{(a_h+1)!(a_{j}-1)!}{a_h!a_j!}\prod_{a\in A} a!=\frac{a_h+1}{a_j}\prod_{a\in A} a!<\prod_{a\in A} a!,$$ a contradiction, hence the result.\end{proof}

Given $N\in \mathcal{M}_{a,b}$, write $R_i(N)$ and $C_i(N)$ (or just $R_i$ and $C_i$) for the $i$th row and column of $N$, respectively. Let $\pi_1$ and $\pi_2$ be the coordinate projections $\S_a\times \S_b\to \S_a$ and $\S_a\times \S_b\to S_b$. Additionally, let $N^*$ denote the multiset of entries of $N$.
\begin{lemma}\label{fixlem}
Let $N\in \mathcal{M}_{a,b}$. If $\sigma=(\rho,\gamma)\in K_N$ satisfies $i^{\rho}=j$, then $R_i^*=R_j^*$. Moreover, if all columns of $N$ are distinct and $H\leq K_N$ is such that $H\pi_1=1$, then $H\pi_2=1$. Both facts hold with the roles of rows and columns swapped.
\end{lemma}
\begin{proof}
The first claim is clear. Suppose the columns of $N$ are distinct and $H\pi_1=1$. Let $(1,\tau)\in H$, and suppose $i^{\tau}=j$. Then $$(n_{1i},n_{2i},\dots,n_{ai})=(n_{1i},n_{2i},\dots,n_{ai})^{(1,\tau)}=(n_{1j},n_{2j},\dots,n_{aj}),$$ but each column is distinct, so $j=i$.
\end{proof}
In Section~\ref{minaut} we shall construct matrices which are nearly cyclic shifts of some tuple. 
Define the \emph{cyclic-shift matrix} of $v\in\mathbb{Z}_{\geq 0}^k$, written $\theta(v)\in\M_{k,k}$, to satisfy $R_i(\theta(v))=v^{(k\,k-1\,\cdots\,1)^i}$ for $0\leq i\leq k-1$. Notice that if $\sum v_i=l$ then $\theta(v)$ is a $(k,l)$-intersection matrix with $|K_{\theta(v)}|\geq k$.

A matrix $N$ is a \emph{truncated staircase} if $N$ is the first $t<k$ rows of $\theta(v)$, where $v=(x,y,z,z,z,\dots,z)\in\Z^k$ for some $x,y\ne z$. A truncated staircase is \emph{weak} if $x=y$, and \emph{strong} if $x\ne y$. For $\gamma\in\S_k$, define $\supp(\gamma):=\{x\in [k] : x^\gamma\ne x\}$.

\begin{lemma}\label{stair}
Let $N$ be a $t\times k$ truncated staircase with $t\geq 2$.\begin{enumerate}
\item[(i)] If $N$ is weak then $[t]\times[t+1]$ is a union of $K_N$-orbits and $K_N|_{[t]\times [t+1]}=\langle \sigma\rangle\cong \mathrm{C}_2$, where $(i,j)^{\sigma}=(t+1-i,t+2-j)$ for all $1\leq i\leq t$ and $1\leq j\leq t+1$. In particular, if $(\rho,\gamma)\in K_N$, then either $\rho=1$, or $j\in\supp(\gamma)$ for each $j\in [t+1]\setminus \{(t+2)/2\}$.
\item[(ii)] If $N$ is strong then $K_N\pi_1=1$.
\end{enumerate}
\end{lemma}
\begin{proof}
If $A=(a_{ij})_{t\times k}$ is a weak truncated staircase, and $B=(b_{ij})_{t\times k}$ a strong truncated staircase, then $K_B\leq K_A$, and moreover, $b_{11}=x\ne b_{t(t+1)}$, whence it suffices to prove the first statement.

 Since $C_1,C_2,\dots, C_{t+1}$ are the only columns which contain an $x$, we deduce from Lemma~\ref{fixlem} that $[t+1]$ is a union of $K_N\pi_2$-orbits.
Therefore, we may assume $k=t+1$. We induct on $t$.

The result is true when $t=2$, so assume $t\geq 3$. Suppose $\sigma=(\rho,\gamma)\in K_N$ is non-trivial. Since $C_1$ and $C_{t+1}$ are the only columns with exactly one $x$, it follows that $\{1,t+1\}^{\gamma}=\{1,t+1\}$. If $1^{\gamma}=1$, then  $\sigma|_{[t]\times ([t]\setminus\{1\})}$ is non-trivial. Similarly if $1^{\gamma}=t+1$, then since $n_{11}$ is the $x$ in $C_1$, and $n_{t(t+1)}$ is the only $x$ in $C_{t+1}$ we deduce that $1^{\rho}=t$, hence $\sigma|_{[t]\times ([t]\setminus\{1\})}$ is non-trivial. Deleting the first and last columns of $N$ and then transposing gives a $(t-1)\times t$ weak truncated staircase $T$. Since $\sigma|_{[t]\times ([t]\setminus\{1\})}$, is non-trivial we deduce that $\sigma^{\top}:=(\gamma,\rho)$ induces a non-trivial permutation in $(\mathrm{Sym}([t]\setminus\{1\})\times \S_{t})_T$, and so, by the inductive hypothesis, $(i,j)^{\sigma^{\top}}=(t+2-i,t+1-j)$ for all $2\leq i\leq t$, and $1\leq j\leq t$. But then, since $n_{11}\ne n_{t1}$ we deduce $(1^{\gamma},(t+1)^{\gamma})=(t+1,1)$, hence the result.
\end{proof}

From here forward, we use $G$ in place of $\S_{k\times l}$ when $k$ and $l$ are clear from context. 
Given a $(k,l)$-bipartite graph $\Gamma$ with adjacency matrix $N=M(\P,\Q)$, we denote the group $\Aut(\Gamma)\cong G_{\P,\Q}$ by $G(N)$.

There is one last class of matrices which we shall frequently encounter in Section~\ref{minaut} --- we show that they yield several vertex automorphisms.
\begin{lemma}\label{centbound}
Let $k\geq3$ and $q$ be positive integers, and $G=\S_{k\times kq}$. Let $N=(n_{ij})_{k\times k}$ be a $(k,qk)$-intersection matrix with largest entry $q+1$, occurring exactly once in each row and column. Then $|K_N|\geq 3$ and  $|G(N)|\geq (q+1)!^kq!^{k^2-2k}(q-1)!^k\cdot 3$.
\end{lemma}
\begin{proof}
For all $i$, since $R_i$ has exactly one entry greater than $q$, and moreover this entry equals $q+1$, it follows that $R_i$ is a permutation of $(q+1,q-1,q,q,\dots,q)$. Since all row and column sums are precisely $qk$, by reordering the rows and columns there is some $\gamma\in \S_k$ of full support such that $n_{ii}=q+1$ and $n_{i(i^\gamma)}=q-1$. Therefore, $(\sigma,\tau)\in K_N$ if and only if $\sigma=\tau$ and for each $i$ there is some $j$ such that $(i^\sigma,i^{\gamma\sigma})=(j,j^{\gamma})$. In other words, $\sigma=\gamma\sigma\gamma^{-1}$, so $K_N\cong C_{\S_k}(\gamma)$. But the centraliser of any element of full support in $\S_k$ has order at least $3$, hence the result.
\end{proof}

\section{Minimising the number of automorphisms}\label{minaut}
In this section we determine the quantity $\chi(k,l)$ from definition \eqref{chi} for each $(k,l)$ with $k\geq 8$ and $l\geq 2$. We are only concerned with generic behaviour (see Section~\ref{probaut}), so the main result is stated for $k\geq 8$, although several individual cases will be shown for smaller $k$. As in Section~\ref{prelim}, we set $G=\S_{k\times l}$ and $K=\S_k\times \S_k$, when $k$ and $l$ are clear.

 We begin by determining $\chi(k,2)$ for $k\geq 3$. The action of $\S_{k\times 2}$ is equivalent to the conjugation action of $\mathrm{S}_{2k}$ on the class $\mathcal{I}$ of fixed-point free involutions.  
This equivalence implies that each 2-point stabiliser in $\S_{k\times 2}$ is the $\S_{2k}$-centraliser of a subgroup generated by two elements of $\mathcal{I}$, hence we are interested in determining the minimum order of such centralisers. Let $\mathrm{D}_{2s}$ denote the dihedral group of order $2s$, with $\mathrm{D}_2=\mathrm{C}_2$ and $\mathrm{D}_4=\mathrm{C}_2^2$.
\begin{proposition}\label{di}
Let $x,y\in \mathcal{I}$ be distinct, and let $H=\langle x,y\rangle$. Then $$C:=C_{\S_{2k}}(H)=\prod_{i=1}^m(\mathrm{D}_{2k_i}\wr \S_{l_i}),$$ for some $m,k_i,l_i\geq 1$ with the $k_i$ distinct and $\sum_{i=1}^m k_il_i=k$. In particular, $\chi(k,2)=2k$ for all $k\geq 3$.
\end{proposition}

\begin{proof}
Build an edge-coloured graph $\Gamma$ with vertex set $[2k]$, red edge $ij$ included if and only if $(i\,j)$ is a transposition in $x$, and blue edge $ij$ included if and only if $(i\,j)$ is a transposition in $y$. Every vertex is incident with exactly one blue and one red edge, and so each connected component is an alternating red-blue cycle of even length. The involution $x$ swaps the endpoints of red edges, and $y$ swaps the endpoints of blue edges, so the connected components of $\Gamma$ are the orbits of $H$ on $[2k]$. 

Let $\Delta_1$ and $\Delta_2$ be connected components of $\Gamma$ of the same size. Then $\Delta_1$ and $\Delta_2$ are isomorphic edge-coloured graphs. Any isomorphism between them is invariant under the actions of $x$ and $y$, whence the actions of $H$ on $\Delta_1$ and $\Delta_2$ are equivalent. 

Since $H=\langle x,y\rangle$, the restriction of $H$ to any component $\Delta_i$ of size $2k_i$ is dihedral of order $2k_i$, acting regularly. Therefore, $$C_{\mathrm{Sym}(\Delta_i)}(H|_{\Delta_i})=C_{\mathrm{Sym}(\Delta_i)}(\mathrm{D}_{2k_i})=\mathrm{D}_{2k_i}.$$ Thus, setting $l_i$ to be the number of components of $\Gamma$ of size $2k_i$, and $m$ the number of distinct component sizes, we deduce $C=\prod_{i=1}^m(\mathrm{D}_{2k_i}\wr \S_{l_i})$ by \cite[Lemma 6.1.8]{ser}.

Suppose that $m=1$, so that $k_1l_1=k$. Then $$|C|=\prod_{i=1}^m(2k_i)^{l_i}l_i!=2k(2k_1)^{l_1-1}(l_1-1)!\in 2k\mathbb{N},$$ and so the smallest possible value when $m=1$ is $2k$, attained when $l_1=1$. Suppose next that $m\geq 2$. Then $\prod_{i=1}^m(2k_i)^{l_i}l_i!\geq 2^m\prod_{i=1}^mk_il_i$. Suppose $k_1l_1=1$, so that $\sum_{i=2}^mk_il_i=k-1$. Then $$2^m\prod_{i=1}^mk_il_i=2^m\prod_{i=2}^mk_il_i\geq2^m\sum_{i=2}^mk_il_i\geq4(k-1)\geq2k,$$ where the first inequality follows from the fact that there is at most one index $j$ such that $l_jk_j=1$. If there is no such $j$, then $2^m\prod_{i=1}^mk_il_i\geq2^m\sum_{i=1}^mk_il_i\geq4k$.
\end{proof}

Let $J_k$ be the $k\times k$ all-one matrix; we omit the subscript when clear from context. If $N$ is any partwise fixed matrix (as defined in Section~\ref{prelim}) then $N+\lambda J$ is also partwise fixed for all $\lambda\in \mathbb{N}$. 
This is key to the results to follow, each of which is stated in the language of $(k,l)$-bipartite graphs, although the proofs are via their bipartite adjacency matrices. Recall from \eqref{autg} that if a $(k,l)$-bipartite graph $\Gamma$ has bipartite adjacency matrix $N$, then $$\Aut(\Gamma)\cong \left(\prod\S_{n_{ij}}\right) :  K_N.$$ Additionally, recall from Section~\ref{prelim} that $A^*$ denotes the multiset of entries of a matrix $A$, and let $\Gamma^{*}=N^{*}$ be the multiset of edge multiplicities of $\Gamma$.
For the remainder of this section let $l>2$ and $k$ be integers, and write $l=qk+r$ where $-2\leq r\leq k-3$.
\begin{proposition}\label{jamesmin}
Suppose $k\geq 8$ and $r\geq 3$. Then $$\chi(k,l)=(q+1)!^{rk}q!^{k^2-rk}.$$ Moreover, if $\Gamma$ is a $(k,l)$-bipartite graph with $|\Aut(\Gamma)|=\chi(k,l)$, then $\Aut_V(\Gamma)=1$ and $$\Gamma^*=\{(q+1)^{rk},(q)^{k^2-rk}\}.$$
\end{proposition}
\begin{proof}
Since $r\geq 3$ and $k\geq \max\{8,r+3\}$, by~\cite[Theorem 1.2]{james} there exists a partwise fixed $(k,r)$-intersection matrix $A$ with all entries in $\{0,1\}$. It follows that $N=A+qJ$ is a partwise fixed $(k,l)$-intersection matrix with all entries $n_{ij}$ in $\{q,q+1\}$. 
Moreover, $\prod_{i,j}n_{ij}!$ is minimum amongst all products of $k^2$ integers with sum $kl$, therefore such a matrix yields a stabiliser of minimum size.
\end{proof}

We now move on to determine $\chi(k,l)$ for each of the remaining values of $r$ separately, starting with the case $r=\pm2$. The proofs take the same general form --- we start with a lemma which explicitly describes a partwise fixed intersection matrix, and then show that any other matrix yields a strictly larger automorphism group, except for a handful of possible exceptions which are dealt with separately. Let $E(i,j)$ be the $k\times k$ matrix with $ij$-entry 1 and all others 0. For the rest of the section we shall assume that $k\geq 7$.

\begin{lemma}\label{bestmat2}
Suppose $r=\pm2$ and set $\varepsilon=r/2$. Let $v=(q+\varepsilon,q+\varepsilon,q,q,\dots,q)\in \Z^k$, and define $$N=\theta(v)-\varepsilon E(k-2,k-2)+\varepsilon E(k-2,k-3)+\varepsilon E(k,k-2)-\varepsilon E(k,k-3).$$ Then $N$ is a partwise fixed $(k,l)$-intersection matrix.
\end{lemma}
\begin{proof}
First note that $\theta(v)$ is a $(k,l)$-intersection matrix, that $N-\theta(v)$ is a $k\times k$ matrix with all row and column sums equal to 0, and that all entries $n_{ij}$ of $N$ are non-negative. So $N$ is a $(k,l)$-intersection matrix. 

Let $(\rho,\gamma)\in K_N$. Since $n_{k(k-3)}$ is the only entry of $N$ equal to $q-\varepsilon$, we deduce from Lemma~\ref{fixlem} that $$k\in\fix(\rho) \text{ and } k-3\in\fix(\gamma).$$ The only entries of $C_{k-3}$ equal to $q+\varepsilon$ are in positions $k-4$, $k-3$, and $k-2$ so we deduce that $\{k-4,k-3,k-2\}$ is setwise stabilised by $\rho$. Suppose $[k-3]^{\rho}=[k-3]$. Then since $N|_{[k-3]\times[k]}$ is a weak truncated staircase, either $\rho|_{[k-3]}=1$, or $k-3\in\mathrm{supp}(\gamma)$ by Lemma~\ref{stair}. The latter is a contradiction, therefore if $\rho\ne1$ then $\{k-1,k-2\}\cap\mathrm{supp}(\rho)\ne\emptyset.$

 If $(k-2)^\rho= k-3$ then since $n_{(k-2)(k-1)}=q+\varepsilon=n_{(k-3)(k-2)}$ and the only other $q+\varepsilon$ in rows $k-2$ and $k-3$ are in $C_{k-3}$ which is fixed, we deduce that $(k-1)^\gamma= k-2$. This is a contradiction as $n_{k(k-1)}=q$ whereas $n_{k(k-2)}=q+\varepsilon$, and $\rho$ fixes $k$. On the other hand, if $(k-2)^\rho = k-4$, then $(k-1)^\gamma =k-4$ since $n_{(k-2)(k-1)}=n_{(k-4)(k-4)}=q+\varepsilon$. But then $n_{(k-5)(k-4)}=q+\varepsilon=n_{(k-1)(k-1)}$ implies that $(k-1)^\rho= k-5$. Finally $n_{(k-5)(k-5)}=q+\varepsilon=n_{(k-1)k}$ so $k^\gamma = k-5$, a contradiction since $n_{kk}=q+\varepsilon$ and $n_{k(k-5)}=q$. But $\rho$ stabilises $\{k-4,k-3,k-2\}$, whence $\rho$ fixes $k-2$, and so $k-1\in \supp(\rho)$.

Since $\rho$ fixes $k$ and $k-2$ we deduce that $(k-1)^\rho\in[k-3]$ . Additionally, $n_{(k-1)(k-1)}=q+\varepsilon$, and so $(k-1)^\gamma\in[k-2]$. 
Moreover, since $k-2\in\fix(\rho)$ and $n_{(k-2)(k-1)},n_{(k-2)(k-3)}$ are the only occurrences of $q+\varepsilon$ in $R_{k-2}$ we deduce $(k-1)^\gamma = k-3$, a contradiction. Therefore, $K_N\pi_1=1$, hence by Lemma~\ref{fixlem} since all columns of $N$ are distinct, $N$ is partwise fixed. 
\end{proof}

\begin{proposition}\label{2}
Suppose $r=\pm2$ and set $\varepsilon=r/2$. Then $$\chi(k,l)=(q-\varepsilon)!q!^{k^2-2k-2}(q+\varepsilon)!^{2k+1}.$$ Moreover, if $\Gamma$ is a $(k,l)$-bipartite graph with $|\Aut(\Gamma)|=\chi(k,l)$, then $\Aut_V(\Gamma)=1$ and $$\Gamma^*=\{(q-\varepsilon)^1,(q)^{k^2-2k-2},(q+\varepsilon)^{2k+1}\}.$$
\end{proposition}
\begin{proof}
Let $N$ be as in Lemma~\ref{bestmat2}. By Lemmas~\ref{top} and~\ref{bestmat2}, \begin{equation}\label{Nsize1} |G(N)|=(q-\varepsilon)!q!^{k^2-2k-2}(q+\varepsilon)!^{2k+1}.\end{equation} In particular if $\Sigma$ is the set of all $(k,l)$-intersection matrices with at least one entry greater than $q$ if $\varepsilon=-1$ and less than $q$ if $\varepsilon=1$, then by applying Lemma~\ref{minN} to $(s,t,X)=(lk,k^2,[(1,q-\varepsilon)])$ (so $x=x(s,t,X)$ satisfies $\{\lfloor x\rfloor,\lceil x\rceil\}=\{q,q+\varepsilon\}$) we deduce that $\min \{|G(A)| : A\in \Sigma\}$ is achieved only by matrices $A$ such that $A^*=N^*$, and hence with $K_A=K_N=1$. 
The proof proceeds by comparing $G(N)$ to the automorphism groups arising from other $(k,l)$-intersection matrices.

Let $A$ be a $(k,l)$-intersection matrix with $|G(A)|\leq |G(N)|$. We shall show that $|G(A)|=|G(N)|$ and that $A^*=N^*$, and hence that $K_A=K_N$. If $A$ has all entries in $\{q,q+\varepsilon\}$, then up to reordering rows and columns, $$A=B+qJ,$$ where $B$ is block diagonal with $s$ blocks and the $i$th block $B_i$ of $B$ is equal to $\theta(u)$ for the vector $u=(\varepsilon,\varepsilon,0,0,0,\dots,0)\in\mathbb{Z}^{m_i}$ for some $m_i\geq 2$. Note that $\prod_{i=1}^s (\S_{m_i}\times\S_{m_i})_{B_i}$ is isomorphic to a subgroup of $K_A$, and so $$|K_A|\geq \prod_{i=1}^s m_i\geq 4.$$ Therefore, by Lemma~\ref{top} and~\eqref{Nsize1}, $$|G(A)|\geq q!^{k^2-2k}(q+\varepsilon)!^{2k}\cdot 4=\frac{4q}{q+1}|G(N)|>|G(N)|,$$ a contradiction, hence $A$ has an entry not in $\{q,q+\varepsilon\}$. We split into cases depending on $\varepsilon$.

\vspace{11pt}\noindent {\underline{\textbf{Case $\varepsilon =-1$:}}} Suppose $\min A^*\leq q-2$. Since $A^*$ has at least one element which is at most $q-2<\lfloor l/k\rfloor$, by applying Lemma~\ref{minN} to $(s,t,X)=(lk,k^2,[(1,q-2)])$ 
we deduce that $$|G(A)|\geq q!^{k^2-2k+1}(q-1)!^{2k-2}(q-2)!=\frac{q^2}{q^2-1}|G(N)|>|G(N)|,$$ a contradiction. Therefore, $A$ has no entry less than $q-1$, and hence $\max A^*\geq q+1$. Thus $A\in \Sigma$, and so by the opening paragraph of the proof $|G(A)|=|G(N)|$. Therefore, by uniqueness $A^*=N^*$, and $A$ is partwise fixed. 

\vspace{11pt}\noindent {\underline{\textbf{Case $\varepsilon =1$:}}} Suppose $\max A^*\geq q+2$. If $A^*$ has more than one element greater than $q+1$ (counting multiplicity), then applying Lemma~\ref{minN} to $(s,t,X)=(lk,k^2,[(2,q+2)])$ we deduce that $$|G(A)|\geq (q+2)!^2q!^{k^2-2k+2}(q+1)!^{2k-4}=\frac{(q+2)^2q}{(q+1)^3}|G(N)|>|G(N)|,$$ a contradiction. Therefore, $A$ has at most one entry greater than $q+1$. Since $|G(N)|$ is minimum over $\Sigma$, it follows that if $\max A^*>q+1$ then $\min A^*= q$. Suppose $A$ has exactly one entry greater than $q+1$, and $\min A^*=q$. Then up to reordering $A$ has first row and column $(q+2,q,q,\dots,q)$, and the submatrix spanning the other rows and columns has the form $B+qJ$ described above. Thus, $$|G(A)|\geq(q+2)!(q+1)!^{2k-2}q!^{k^2-2k+1}\cdot 4>|G(N)|,$$ a contradiction. Therefore $|G(A)|=|G(N)|$, $A^*=N^*$, and $A$ is partwise fixed. 
\end{proof}
We next consider the case $r=0$, proceeding in a similar fashion.

\begin{lemma}\label{bestmat0}
Suppose $r=0$. Let $v=(q+1,q-1,q,q,q,\dots,q)\in \Z^k$, and define $$N=\theta(v)+E(k-3,k)-E(k-3,1)-E(k,k)+E(k,1).$$ Then $N$ is a partwise fixed $(k,l)$-intersection matrix.
\end{lemma}
\begin{proof}

By inspection $N$ is a $(k,l)$-intersection matrix. Let $(\rho,\gamma)\in K_N$. 
Note that $N|_{[k-4]\times [k]}$ is a strong truncated staircase, so by Lemma~\ref{stair} either $\rho$ does not setwise stabilise $[k-4]$, or $[k-4]\subseteq\fix(\rho)$. Moreover, by Lemma~\ref{fixlem}, since $R_{k-3}$ is the only row with two entries equal to $q+1$, and $R_k$ is the only row with all entries equal to $q$, we deduce that $\rho$ fixes $k-3$ and $k$. Since $\rho$ fixes $k-3$ and $R_{k-3}=(q-1,q,q,\dots,q,q+1,q-1,q,q+1)$, we deduce \begin{equation}\label{stabsets}\{1,k-2\}^\gamma=\{1,k-2\}\quad\text{and}\quad\{k-3,k\}^\gamma=\{k-3,k\}.\end{equation} Moreover, if $\rho\ne 1$  then $\{k-2,k-1\}\cap\supp(\rho)\ne\emptyset$.

Suppose $(k-2)^\rho= i$ for some $i\not\in\{k-3,k-2,k\}$. Then since $n_{(k-2)(k-2)}$ is the only $q+1$ in $C_{k-2}$, it follows that $k-2\in\supp(\gamma)$, hence $(k-2)^\gamma=1=(k-2)^\rho$ by \eqref{stabsets}. Now, $n_{12}=q-1=n_{(k-2)(k-1)}$, so $(k-1)^\gamma= 2$, and so $n_{22}=q+1=n_{(k-1)(k-1)}$ implies that $(k-1)^\rho= 2$. But $n_{23}=q-1=n_{(k-1)k}$ so we conclude $k^\gamma= 3$, a contradiction. Therefore, $\rho$ fixes $k-2$.

Finally, suppose $(k-1)^\rho=i\not\in \{k-3,k-2,k-1,k\}$. Then since the only $q+1$ in $C_{k-1}$ is $n_{(k-1)(k-1)}$, we deduce that $k-1\in\supp(\gamma)$. Similarly, the only $q-1$ in $R_{k-2}$ is $n_{(k-2)(k-1)}$, hence $k-2\in\supp(\rho)$, a contradiction. Therefore $\rho=1$. The result now follows from Lemma~\ref{fixlem}.
\end{proof}
\begin{proposition}\label{0}
Suppose $r=0$. Then $$\chi(k,l)=(q+1)!^kq!^{k^2-2k}(q-1)!^k.$$ Moreover, if $\Gamma$ is a $(k,l)$-bipartite graph with $|\Aut(\Gamma)|=\chi(k,l)$, then either $\Aut_V(\Gamma)=1$ and $\Gamma^*=\{(q+1)^k,(q)^{k^2-2k},(q-1)^k\}$, or $l=k$, $|\Aut_V(\Gamma)|=2$, and $\Gamma^*=\{(2)^{k-1},(1)^{k^2-2k+2},(0)^{k-1}\}$.
\end{proposition}
\begin{proof}
Let $N$ be as in Lemma~\ref{bestmat0}, so that \begin{equation}\label{G(N)2}|G(N)|=(q+1)!^kq!^{k^2-2k}(q-1)!^k,\end{equation} by Lemma~\ref{top}. Let $A$ be a $(k,l)$-intersection matrix with $|G(A)|\leq |G(N)|$.  We shall show that $|G(A)|=|G(N)|$, and that either $A^*=N^*$ (and hence $K_A=K_N$), or $k=l$ and $|K_A|=2$.

Let $m$ be the number of entries of $A$ which are greater than $q$. It follows from \eqref{G(N)2} and Lemma~\ref{minN} applied to $(s,t,X)=(lk,k^2,[(m,q+1)])$ that $|G(N)|$ is minimum amongst all $(k,l)$-intersection matrices with $m\geq k$, and $N^*$ is the unique multiset achieving this minimum, hence the desired result holds for $m\geq k$. Suppose $m<k$. Then by permuting rows and columns we may assume that the initial $m\times m$ submatrix of $A$ is an $(m,qm)$-intersection matrix with $m$ entries greater than $q$. Moreover, all entries outside of this submatrix are $q$. Therefore, $K_A$ induces $\S_{k-m}\times \S_{k-m}$ on the final $k-m$ rows and columns, hence by Lemmas~\ref{top} and~\ref{minN} \begin{equation}\label{Abound}|G(A)|\geq (q+1)!^mq!^{k^2-2m}(q-1)!^m(k-m)!^2=\frac{q^{k-m}(k-m)!^2}{(q+1)^{k-m}}|G(N)|.\end{equation} The right-hand side of \eqref{Abound} is greater than $|G(N)|$ if $m<k-2$, so $m\geq k-2$. We now split into two cases.

\vspace{11pt}\noindent \textbf{\underline{Case $m=k-2$:}} Here, the right-hand side of \eqref{Abound} is $4q^2|G(N)|/(q+1)^2$, which implies that $q=1$, and $|G(A)|=|G(N)|$. Moreover, $\max A^*=2$, by the uniqueness statement of Lemma~\ref{minN}. If each of the $k-2$ entries equal to 2 are in distinct rows and columns, then $A$ has first $m\times m$ submatrix as in Lemma~\ref{centbound}, hence $$|G(A)|=|G(A|_{[m]\times [m]})|\cdot 2!^2\geq (2^{k-2}\cdot3)\cdot2!^2=3\cdot2^{k}>2^k=|G(N)|,$$ a contradiction, thus some row or column has two occurrences of 2. 
Therefore, $A$ has at least three columns and two rows equal to $(1)^k$ or vice versa. In particular, $|K_A|\geq 3!2!$. Therefore, $|G(A)|\geq 2^{k-2}\cdot3!\cdot2= 3\cdot2^{k}>|G(N)|$, a contradiction.

\vspace{11pt}\noindent \textbf{\underline{Case $m=k-1$:}} Suppose first that $\max A^*>q+1$. From Lemma~\ref{minN} applied to the triple $(s,t,X)=(lk,k^2,[(1,q+2),(k-2,q+1)])$ we deduce that $$|G(A)|\geq (q+2)!(q+1)!^{k-2}q!^{k^2-2k+1}(q-1)!^{k}=\frac{q+2}{q+1}|G(N)|>|G(N)|,$$ a contradiction. Therefore, the $m$ largest entries of $A$ are all $q+1$. If each of these $k-1$ entries are in distinct rows and columns, then $A$ has initial $(k-1)\times (k-1)$ submatrix as in Lemma~\ref{centbound}, hence  $$|G(A)|\geq ((q+1)!^{k-1}q!^{k^2-4k+3}(q-1)!^{k-1}\cdot 3)q!^{2k-1}=\frac{3q}{q+1}|G(N)|> |G(N)|.$$
Therefore, some row or column contains two ($q+1$)s. Thus, $A$ either has at least two rows or two columns which are $(q)^k$, and so $|K_A|\geq 2$, whence $$|G(A)|\geq (q+1)!^{k-1}q!^{k^2-2k+2}(q-1)!^{k-1}\cdot2=\frac{2q}{q+1}|G(N)|.$$ 
But $|G(N)|\geq |G(A)|$, so we deduce that $|G(A)|=|G(N)|$, that $q=1$, and that $|K_A|=2$.
\end{proof}
We now define our matrices for the case $r=\pm1$.
Suppose $k$ is odd. Set $u=(q-r,q,q,q,\dots,q,q+r)\in\Z^{\lceil k/2\rceil}$, $v=(q+r,q,q,\dots,q)\in\Z^{\lfloor k/2\rfloor}$, and $w=(q+r,q,q,\dots,q)\in\Z^{\lceil k/2\rceil}$. Let $N_1:=\theta(u)+r E(1,2)$, let $N_2$ be the $\lceil k/2\rceil\times \lfloor k/2\rfloor$ matrix with first row $(q)^{\lfloor k/2\rfloor}$, and the remaining $\lfloor k/2\rfloor$ rows given by $\theta(v)$, and $N_3$ the $\lfloor k/2\rfloor\times\lceil k/2\rceil$ matrix given by removing the second row of $\theta(w)$. Define $$B(k,r):=\begin{pmatrix} N_1&\aug& N_2\\ \hline N_3&\aug& qJ\\\end{pmatrix}.$$ For $k$ even, let $C(k,r)$ be the matrix obtained from $B(k-1,r)$ by appending $(q,q,\dots,q,q+r)$ as the last row and column.
\begin{lemma}\label{bestmat1}
Suppose $r=\pm 1$. Let $$N=\begin{cases} B(k,r)&\text{ if $k$ is odd}\\C(k,r)&\text{ if $k$ is even.}\end{cases}$$ Then $N$ is a partwise fixed $(k,l)$-intersection matrix.
\end{lemma}
\begin{proof}
A quick check verifies that $N$ is indeed a $(k,l)$-intersection matrix. Let $(\rho,\gamma)\in K_N$. Since $R_i$ has an entry equal to $q-r$ if and only if $i\leq \lceil k/2\rceil$, it follows from Lemma~\ref{fixlem} that both $[\lceil k/2\rceil ]$ and $[k]\setminus[\lceil k/2\rceil]$ are unions of $K_N\pi_1$-orbits. Moreover, each $R_i$ with $i>\lceil k/2\rceil$ has exactly one entry equal to $q+r$, and each such entry occurs in a distinct column (at most one of which has index greater than $\lceil k/2\rceil$) so if $\lceil k/2\rceil<i\in\supp(\rho)$ then there is some $c\in\supp(\gamma)$ for $c\leq \lceil k/2\rceil$. The same is true when the roles of rows and columns are reversed. That is, if $\supp(\rho,\gamma)\cap\left( [k]^2\setminus [\lceil k/2\rceil ]^2\right)\ne\emptyset$ then $\supp(\rho,\gamma)\cap[\lceil k/2\rceil ]^2\ne\emptyset$, hence it suffices to show that $H:=(\S_{\lceil k/2\rceil}\times\S_{\lceil k/2\rceil})_{N_1}=1$.

Let $(\sigma,\tau)\in H$. By Lemma~\ref{fixlem}, $1\in\fix(\sigma)$ as $R_1$ is the only row with two entries equal to $q+r$. Moreover, $N_1|_{([\lceil k/2\rceil]\setminus\{1\})\times [\lceil k/2\rceil]}$ is a strong truncated staircase so we deduce from Lemma~\ref{stair} that $\sigma=1$. Since all columns of $N_1$ are distinct, $H=1$ by Lemma~\ref{fixlem}.
\end{proof}

\begin{proposition}\label{1}
Suppose $r=\pm1$. Then $$\chi(k,l)=(q-r)!^{\lceil k/2\rceil}q!^{k^2-k-2\lceil k/2\rceil}(q+r)!^{k+\lceil k/2\rceil}.$$ Moreover, if $\Gamma$ is a $(k,l)$-bipartite graph with $|\Aut(\Gamma)|=\chi(k,l)$, then either $\Aut_V(\Gamma)=1$ and $$\Gamma^*=\{(q-r)^{\lceil k/2\rceil},(q)^{k^2-k-2\lceil k/2\rceil},(q+r)^{k+\lceil k/2\rceil}\},$$ or $l=k+r$, $|\Aut_V(\Gamma)|=2$, and $$\Gamma^*=\{(1-r)^{\lceil k/2\rceil-1},(1)^{k^2-k-2\lceil k/2\rceil+2},(1+r)^{k+\lceil k/2\rceil-1}\}.$$
\end{proposition}
\begin{proof}
Let $N$ be as in Lemma~\ref{bestmat1}. From Lemma~\ref{top}, we deduce that \begin{equation}\label{G(N)1}|G(N)|=(q-r)!^{\lceil k/2\rceil}q!^{k^2-k-2\lceil k/2\rceil}(q+r)!^{k+\lceil k/2\rceil}.\end{equation} 
Let $A=(a_{ij})$ be a $(k,qk+r)$-intersection matrix with $|G(A)|\leq |G(N)|$. 

If $r=1$ then let $m$ be the number of entries of $A$ less than $q$, and if $r=-1$ then let $m$ be the number of entries of $A$ greater than $q$. If $m\geq\lceil k/2\rceil$, then by \eqref{G(N)1} and Lemma~\ref{minN} applied to $(s,t,X)=(lk,k^2,[(\lceil k/2\rceil,q-r)])$, since $$ x(s,t,X)=\frac{lk-(\lceil k/2\rceil)(q-r)}{k^2-\lceil k/2\rceil}$$ is between $q$ and $q+r$ we deduce that $|G(A)|=|G(N)|$, that $A$ is partwise fixed and that $A^*=N^*$. Therefore, we may assume $m<\lceil k/2\rceil$.

By reordering the rows and columns so that the $m$ distinguished entries occur in the initial $m\times m$ submatrix, $A$ has the form
$$\begin{pmatrix}
A_1&\aug&A_2\\\hline
A_3&\aug&A_4\\
\end{pmatrix}$$
where $A_1$ is $m\times m$ and $A_2,A_3,A_4$ have all entries in $\{q,q+r\}$  (for if not then they would also contain one of the distinguished entries, contradicting that there are only $m$ such entries). In particular, since all row and column sums are $qk+r$ we deduce that $q+r$ occurs exactly once in $R_i$ and $C_i$ for each $i>m$. For $1\leq i\leq m$, let $\mathcal{C}_i$ 
 be the set of all $j\in \{m+1,m+2,\dots, k\}$ such that $a_{ij}=q+r$ --- these entries appear in $A_2$ --- and similarly $\mathcal{R}_i$ 
 the set of all $ j\in \{m+1,m+2,\dots, k\}$ such that $a_{ji}=q+r$ --- these entries appear in~$A_3$. 

If $j_1,j_2\in \mathcal{C}_i$ for $1\leq i\leq m$, then $(1, (j_1\, j_2))\in K_A$, and similarly if $j_1,j_2\in \mathcal{R}_i$ then $((j_1\, j_2),1)\in K_A$. Finally, if $j_1,j_2\in \mathcal{C}_0$, then there exist unique $i_1\ne i_2>m$ such that $a_{i_1j_1}=a_{i_2j_2}=q+r$ and so $((i_1\, i_2), (j_1\, j_2))\in K_A$. Combining this with Lemma~\ref{minN} applied to $(s,t,X)=(lk,k^2,[(m,q-r)])$ gives \begin{equation}\label{asize}|G(A)|\geq (q-r)!^mq!^{k^2-k-2m}(q+r)!^{k+m}|\mathcal{C}_0|!\prod_{i=1}^{m}|\mathcal{C}_i|!|\mathcal{R}_i|!.\end{equation} We now split into cases according to the value of $m$.

\vspace{11pt}\noindent \underline{\textbf{Case $m<\lceil k/2\rceil -1$:}} The matrix $A_2$ has $k-m$ columns and at most one $q+r$ per column, so at most $m+1$ columns are distinct. Therefore, $|\mathcal{C}_0|!\prod_{i=1}^{m}|\mathcal{C}_i|!|\mathcal{R}_i|!\geq 2^{k-2m-1}$, so by \eqref{asize} $$|G(A)|\geq 2^{k-2m-1}(q-r)!^mq!^{k^2-k-2m}(q+r)!^{k+m}=\frac{2^{k-2m-1}q^{\lceil k/2\rceil-m}}{(q+1)^{\lceil k/2\rceil-m}}|G(N)|>|G(N)|,$$ a contradiction.

\vspace{11pt}\noindent \underline{\textbf{Case $m=\lceil k/2\rceil -1$:}} If two columns of $A_2$ or two rows of $A_3$ are the same, then by \eqref{asize}, $$|G(A)|\geq2(q-r)!^mq!^{k^2-k-2m}(q+r)!^{k+m}=\frac{2q}{q+1}|G(N)|\geq |G(N)|,$$ with equality only possible if $q=1$; $|K_A|=2$; and $m$ of the entries are equal to $(1-r)$, $k+m$ entries to $(1+r)$, and all other entries $1$. 
 On the other hand if all columns of $A_2$ are distinct, as well as all rows of $A_3$, then each row of $A_2$ has sum $(k-m)q+1$, as does each column of $A_3$, hence $A_1$ is an $(m,qm)$-intersection matrix. If $A$ is not partwise fixed then as before $|G(A)|\geq\frac{2q}{q+1}|G(N)|.$ Therefore, there remain two possibilities for $A$:
\begin{itemize}
    \item[(i)] $q=1$, $A_1$ is an $(m,m)$-intersection matrix with $m$ entries equal to $1-r$, all other entries in $\{1+r,1\}$,  and $|K_A|=2$.
    \item[(ii)] $A$ is partwise fixed with $A_1$ a partwise fixed $(m,qm)$-intersection matrix, all columns of $A_2$ are distinct, and all rows of $A_3$ are distinct.

\end{itemize}
In the first case we are done, so consider case (ii) --- we show it is not possible. If the $m$ distinguished entries of $A_1$ are not all in distinct rows and columns, then without loss of generality $R_{m}(A)$ begins with $(q)^m$. Since $m=\lceil k/2\rceil-1$ and all rows of $A_3$ are distinct, some $R_i(A)$ with $m+1\leq i\leq k$ begins $(q)^m$. Let $j_1,j_2\geq m+1$ be the unique indices such that $a_{mj_1}=a_{ij_2}=q+r$. Then $((m\, i), (j_1\,j_2))\in K_A$, and so $|K_A|\geq2$, a contradiction, whence the distinguished entries of $A_1$ are in distinct rows and columns.

Suppose $A_1$ has an entry greater than $q+1$ if $r=-1$ or less than $q-1$ if $r=1$. Then from Lemma~\ref{minN} we deduce that $$|G(A)|\geq (q-2r)!(q-r)!^{m-1}q!^{k^2-k-2m-1}(q+r)!^{k+m+1}=\frac{q-r+1}{q-r}|G(N)|>|G(N)|,$$
a contradiction, hence the $m$ distinguished entries all equal $q-r$. Since each $q-r$ occurs in a distinct row and column, $q+sr\not\in A_1^*$ for any $s>1$ --- that is, $A_1$ is as in Lemma~\ref{centbound}, so is not partwise fixed, a contradiction.
\end{proof}
Putting together Propositions~\ref{di},~\ref{jamesmin},~\ref{2},~\ref{0}, and~\ref{1}, we deduce Theorem~\ref{mainmin}.

\section{Asymptotics}\label{probaut}
In this section we are concerned with the limiting distributions arising from uniformly sampling $(k,l)$-bipartite graphs. 
Most of the proofs in this section will be via intersection matrices, which can be thought of as vertex-labelled $(k,l)$-bipartite graphs (see Section~\ref{prelim}); the labellings are usually irrelevant to the asymptotics. As our study is mostly through matrices, some of the results and discussion here can be viewed as an analysis of the integer points of special classes of \emph{transportation polytopes} --- generalisations of the famous Birkhoff polytope. Given positive integers $k$ and $l$, define $\Omega_{k,l}$ to be the set of all $(k,l)$-intersection matrices, and $H_k(l):=|\Omega_{k,l}|.$ We first state a marvellous result of Stanley, adapted to suit our use throughout this section. 

\begin{theorem}[\cite{stanley}]\label{iklsize}
Let $k,l\geq2$ be integers. Then $H_k(l)$ is a polynomial in $l$ of degree $(k-1)^2$. In particular, for each $k$ there is some $c(k)>0$ depending only on $k$ such that $H_{k}(l)\geq c(k)l^{(k-1)^2}$ for all $l$.
\end{theorem}
Throughout this section it will be useful to recall that $\Aut_E(\Gamma)$ is the group of automorphisms of the graph $\Gamma$ which fix the vertices pointwise, and that the full automorphism group of a $(k,l)$-bipartite graph is given up to isomorphism by \eqref{autg}. We begin our analysis by considering the vertex automorphisms --- there will almost always be as few as possible.
\begin{proposition}\label{trivverts}
Let $k\geq 2$, let $l\to\infty$, and let $\Gamma$ be a uniformly randomly chosen $(k,l)$-bipartite graph. Then asymptotically almost surely $$|\Aut_V(\Gamma)|=\begin{cases}2&\text{ if $k=2$,}\\
1&\text{ otherwise.}\end{cases}$$
\end{proposition}
\begin{proof}
We show the analogous result for intersection matrices --- that is, for a uniformly chosen $(k,l)$-intersection matrix $N$, asymptotically almost surely $|K_N|=2$ if $k=2$ and $K_N=1$ otherwise. From this the desired result follows since each $(k,l)$-bipartite graph corresponds to at most $k!^2$ such matrices. 

If $N$ is a $(2,l)$-intersection matrix then $N=\theta(a,l-a)$ for some $a$, hence $|K_N|\ne 2$ if and only if $N=\frac{l}{2}J$. Since there are exactly $l$ choices for $a$ the result follows in this case.

Consider now $k\geq 3$. We give a na\"ive upper bound on the number of $(k,l)$-intersection matrices with at least one row being a permutation of another, which we call \emph{bad}, and show that these make up a small proportion of all $(k,l)$-intersection matrices.

Each bad matrix can be constructed as follows. We first pick two rows to have the same multisets of entries --- there are ${k\choose 2}$ ways of doing so, call them $R_1$ and $R_2$. There are at most $l^{k-1}$ possibilities for each of $R_1$, $R_3$, $R_4,\dots, R_{k-1}$ (as the final entry of each is determined by the first $k-1$), and then at most $k!$ choices for $R_2$. Finally $R_{k}$ is uniquely determined by $R_1,R_2,\dots, R_{k-1}$, hence there are at most ${k\choose 2}k!l^{(k-1)(k-2)}$ bad matrices; this quantity is $o(H_k(l))$ by Theorem \ref{iklsize}. By symmetry the same is true when columns are considered instead of rows, hence $K_N=1$ by Lemma~\ref{fixlem}, as desired. 
\end{proof}

To prove Theorem~\ref{mainrand} we will make use of a very nice identity, which seems to be known, but a proof in full generality is difficult to locate --- we sketch one here.

\begin{lemma}\label{genrec}
Let $f:\mathbb{N}\to\mathbb{N}$ be a function, set $b_0=1$ and for each positive integer $n$ define $b_n:=\sum_{\lambda\vdash n}\prod_{j}f(\lambda_j)$, where the sum is over all partitions $\lambda=(\lambda_1,\dots,\lambda_t)$ of $n$. Then $$b_n=n^{-1}\sum_{t=1}^n\left(b_{n-t}\sum_{d|t}d\cdot f(d)^{t/d}\right).$$
\end{lemma}
\begin{proof}
Note that $b_n$ has generating function $\sum_{n\geq 0}b_nx^n=\prod_{d\geq 1}(1-f(d)x^d)^{-1}$. By taking logarithms of both sides, followed by differentiating and then multiplying by $x$ we deduce the equality $$\frac{\sum_{n\geq 0} nb_nx^n}{\sum_{i\geq 0}b_ix^i}=\sum_{d\geq 1}\frac{df(d)x^d}{1-f(d)x^d}=\sum_{d\geq 1}d(f(d)x^d+f(d)^2x^{2d}+\cdots).$$ The coefficient of $x^n$ in $$\left(\sum_{i\geq 0}b_ix^i\right)\left(\sum_{d\geq 1}d(f(d)x^d+f(d)^2x^{2d}+\cdots)\right)$$ is $\sum_{t=1}^nb_{n-t}\sum_{d|t}d\cdot f(d)^{t/d}$, hence the result.
\end{proof}
\begin{remark}
A special case of this result appears in Erd\H{o}s' beautiful 1942 paper on the asymptotics of integer partitions \cite{erdos}; his proof is somewhat more direct.
\end{remark}

For our purposes we restrict ourselves to $b_n:=\sum_{\lambda\vdash n}\prod_{j}\lambda_j!$, that is, choose the function in Lemma~\ref{genrec} to be $f(d)=d!$. We will use this sequence to approximate the products of the entry factorials of our $(k,l)$-intersection matrices in a natural way. The next lemma proves useful in bounding $b_n$.

\begin{lemma}\label{prodmin}
Let $n\geq 17$ be an integer. Then $$(n-t)!t!\left(1+\frac{4}{n-t}\right)\cdot2t\leq (n-1)!\cdot 3$$ for all  $1\leq t\leq n-3$. 
\end{lemma}
\begin{proof}
The result holds when $t=1$.  If $2\leq t\leq n-8$, then \begin{equation}\label{factineq}(n-t)!t!\left(1+\frac{4}{n-t}\right)\cdot2t\leq3(n-t)!t!\cdot t\leq 3(n-t)!(t+1)!.\end{equation} Now, $(n-t)!(t+1)!$ is non-increasing in $t$ until $t=\lceil (n+1)/2\rceil$, at which point it is non-decreasing, and moreover $(n-t)!(t+1)!$ is symmetric about $(n+1)/2$, whence it suffices to show that when $t=2$ the right hand side of \eqref{factineq} is at most $(n-1)!\cdot 3$. Evaluating the expression we deduce $3(n-2)!3!=\frac{6}{n-1}(n-1)!\cdot 3\leq (n-1)!\cdot 3$, as desired.

For $t=n-3$ we calculate \begin{align*}(n-(n-3))!(n-3)!\left(1+\frac{4}{n-(n-3)}\right)\cdot2(n-3)&=6\cdot(n-3)!\left(1+\frac{4}{3}\right)\cdot 2(n-3)\\&=28(n-3)!(n-3)\\&=\frac{28(n-3)}{3(n-1)(n-2)}(n-1)!\cdot 3\\&\leq (n-1)!\cdot 3\end{align*} since $n\geq 17$. A similar calculation can be done for $n-7\leq t\leq n-4$, so the result follows.
\end{proof}
\begin{lemma}\label{comps}
Let $n\geq 1$ and let $b_n=\sum_{\lambda\vdash n}\prod_{j}\lambda_j!$. Then $b_n\leq n!\left(1+\frac{4}{n}\right)$.
\end{lemma}

\begin{proof}
By Lemma~\ref{genrec} with $f(d)=d!$ \begin{equation}\label{rec}b_n=n^{-1}\sum_{t=1}^n\left( b_{n-t}\sum_{d|t}d\cdot (d!)^{t/d}\right).\end{equation} We start by giving an upper bound on $\sum_{d|t}d(d!)^{t/d}$ for each $t\geq 1$, and then induct on $n$.

First, $\sum_{d|t}d\cdot (d!)^{t/d}=t\cdot t!+\sum_{d|t,d<t}d\cdot (d!)^{t/d}$. If $d|t$ with $d<t$, then $d\leq t/2$, and moreover $(d!)^{t/d}$ is a product of each integer less than or equal to $d$, each occurring in the product exactly $t/d$ times. On the other hand $\left\lfloor\frac{t}{2}\right\rfloor!\left\lceil\frac{t}{2}\right\rceil!$ is a product of $2d$ of the aforementioned integers with an additional $t-2d$ integers each of which is greater than $d$. Therefore, $$d\cdot(d!)^{t/d}\leq \frac{t}{2}\left\lfloor\frac{t}{2}\right\rfloor!\left\lceil\frac{t}{2}\right\rceil!.$$
Moreover, $t$ has at most $2t^{1/2}$ divisors, whence $\sum_{d|t,d<t}d\cdot (d!)^{t/d}\leq 2t^{1/2}\frac{t}{2}\left\lfloor\frac{t}{2}\right\rfloor!\left\lceil\frac{t}{2}\right\rceil!$, and so $$\sum_{d|t}d\cdot (d!)^{t/d}\leq t\cdot t!+t^{3/2}\left\lfloor\frac{t}{2}\right\rfloor!\left\lceil\frac{t}{2}\right\rceil!= t\cdot t!\left(1+t^{1/2}{t\choose \lfloor t/2\rfloor}^{-1}\right).$$ Finally, ${t\choose \lfloor t/2\rfloor}\geq\left(\frac{t}{\lfloor t/2\rfloor}\right)^{\lfloor t/2\rfloor}\geq\left(\frac{t}{t/2}\right)^{(t-1)/2}=2^{(t-1)/2}$ for all $t>1$, and $1^{1/2}{1\choose 0}=1$, whence \begin{equation}\label{dbound}\sum_{d|t}d\cdot (d!)^{t/d}\leq \left(1+\frac{t^{1/2}}{2^{(t-1)/2}}\right)t\cdot t!\end{equation}\begin{equation}\label{dbound2}\leq 2t\cdot t!.\end{equation}

We now induct on $n$. The claim can be verified computationally for $1\leq n\leq 16$, so suppose $n\geq 17$. Set $R_n:=\frac{1}{n\cdot n!}\sum_{t=1}^{n-3}\left(b_{n-t}\sum_{d|t}d\cdot (d!)^{t/d}\right)$. Then 
\begin{align*}R_n&\leq \frac{1}{n\cdot n!}\sum_{t=1}^{n-3}b_{n-t} 2t\cdot t!&&\text{by~\eqref{rec} and~\eqref{dbound2}}\\&\leq\frac{1}{n\cdot n!}\sum_{t=1}^{n-3}(n-t)!\left(1+\frac{4}{n-t}\right)2t\cdot t!&&\text{by induction}\\&\leq \frac{n-3}{n\cdot n!}3(n-1)!&&\text{by Lemma~\ref{prodmin}}\\&=3(n-3)/n^2.\end{align*}
 
Again by~\eqref{rec}, $$b_n=n^{-1}\left(b_0\left(\sum_{d|n}d\cdot (d!)^{n/d}\right)+b_1\left(\sum_{d|(n-1)}d\cdot (d!)^{(n-1)/d}\right)+b_2\left(\sum_{d|(n-2)}d\cdot (d!)^{(n-2)/d}\right)\right)+n!R_n.$$
Applying~\eqref{dbound} we deduce that $b_n$ is at most $$n!\left(b_0\left(1+\frac{n^{1/2}}{2^{(n-1)/2}}\right)+b_1\cdot \frac{n-1}{n^2}\left(1+\frac{(n-1)^{1/2}}{2^{(n-2)/2}}\right)+b_2\cdot \frac{n-2}{n^2(n-1)}\left(1+\frac{(n-2)^{1/2}}{2^{(n-3)/2}}\right)+R_n\right).$$Now, $b_0=b_1=1$, and $b_2=3$, so putting this together with $R_n\leq 3(n-3)/n^2$ we deduce that \begin{align*}b_n&\leq n!\left(1+\frac{n^{1/2}}{2^{(n-1)/2}}+\frac{n-1}{n^2}\left(1+\frac{(n-1)^{1/2}}{2^{(n-2)/2}}\right)+\frac{3n-6}{n^2(n-1)}\left(1+\frac{(n-2)^{1/2}}{2^{(n-3)/2}}\right)+\frac{3(n-3)}{n^2}\right)\\&=n!(1+g(n)),\end{align*} say. 
One can verify that $g(n)<4/n$ for $17\leq n\leq 25$, and standard techniques (e.g. the first derivative test) show that for $n\geq 25$, $g(n)/(4/n)$ is increasing with limit 1, hence the result.
\end{proof}

\begin{proposition}\label{factexp}
Let $k$ and $l$ be integers at least $2$ and let $X$ be a uniformly chosen $(k,l)$-intersection matrix. 
Then there exists a constant $d_0(k)>0$ depending only on $k$ such that $$\mathbb{E}\left(\prod_{i,j\leq k} X_{i,j}!\right)\leq\frac{d_0(k)}{l^{(k-1)^2}}l!^k.$$
\end{proposition}
\begin{proof}
Since $X$ is chosen uniformly, $$\mathbb{E}\left(\prod_{i,j\leq k} X_{i,j}!\right)=H_{k}(l)^{-1}\sum_{N\in\Omega_{k,l}}\left(\prod_{i,j}n_{ij}!\right)
$$ 
Let $N\in\Omega_{k,l}$. Then each row of $N$ is an ordered partition of $l$ into exactly $k$ non-negative parts, thus we deduce from Lemma~\ref{comps} that 
$$\sum_{N\in \Omega_{k,l}}\prod n_{ij}!\leq (k!b_l)^k\leq \left(k!l!\left(1+\frac{4}{l}\right)\right)^k.$$
Therefore, $$\mathbb{E}\left(\prod_{i,j\leq k} X_{i,j}!\right)\leq H_{k}(l)^{-1}\left(k!l!\left(1+\frac{4}{l}\right)\right)^k\leq H_{k}(l)^{-1}l!^k3^kk!^k.$$ But $H_{k}(l)\geq c(k)l^{(k-1)^2}$ by Theorem~\ref{iklsize}, whence there is some $d_0(k)>0$ with the desired property.
\end{proof}

We now translate our work into the language of $(k,l)$-bipartite graphs.
\begin{proposition}\label{bipexp}
Let $k$ and $l$ be integers at least $2$ and let $\Gamma$ be a uniformly chosen $(k,l)$-bipartite graph. Then there exists a constant $d(k)>0$ depending only on $k$ such that $$\mathbb{E}(|\Aut_{E}(\Gamma)|)\leq\frac{d(k)}{l^{(k-1)^2}}l!^k.$$
\end{proposition}

\begin{proof}
Let $X$ be a uniformly chosen $(k,l)$-intersection matrix. Each $(k,l)$-bipartite graph, $\Sigma$, corresponds to at most 
$k!^2$ distinct $(k,l)$-intersection matrices. 
Thus if $N$ is a $(k,l)$-intersection matrix and $\Sigma$ is a $(k,l)$-bipartite graph with bipartite adjacency matrix $N$, then $\mathbb{P}(\Gamma=\Sigma)/\mathbb{P}(X=N)\leq k!^2.$ Therefore, $$\mathbb{E}(|\Aut_{E}(\Gamma)|)\leq k!^2\mathbb{E}\left(\prod X_{i,j}!\right).$$ The result follows from Proposition~\ref{factexp}.
\end{proof}
Finally we can deduce Theorem~\ref{mainrand}
\begin{proof}[Proof of Theorem~\ref{mainrand}]
We show that $$|\Aut_E(\Gamma)| \leq l^{\varepsilon-{(k-1)}^2}l!^k$$ asymptotically almost surely: combining this with Proposition~\ref{trivverts} gives the result.

By Markov's inequality, $$\mathbb{P}\left(|\Aut_E(\Gamma)|\geq l^{\varepsilon-{(k-1)}^2}l!^k\right)\leq \frac{\mathbb{E}(|\Aut_{E}(\Gamma)|)}{l^{{(k-1)}^2-\varepsilon}l!^k}.$$ By Proposition~\ref{bipexp}, the right hand side tends to 0 as $l$ gets large.
\end{proof}

To understand the distribution of entries of $(k,l)$-intersection matrices we use a framework similar to that of Chatterjee, Diaconis, and Sly~\cite{cds}. We begin by setting up some notation.
Let $k\geq 2$, and $l\geq 1$. We use $|_{k-1}$ to denote the restriction of a $k\times k$ matrix to its first $k-1$ rows and columns. Given a matrix $A\in\mathbb{Z}^{(k-1)\times (k-1)}$, there is a unique matrix $\varphi_l(A)\in\Z^{k\times k}$ with all row and column sums equal to $l$, and such that $\varphi_l(A)|_{k-1}=A$. Indeed, there is clearly a unique choice for each entry in the $k$th row or column aside from the $(k,k)$ entry. Moreover, $\sum_{i=1}^{k-1} \varphi_l(A)_{ik}=\sum_{i=1}^{k-1} \varphi_l(A)_{ki}$ since both sums are uniquely determined by the sum of all entries of $A$, whence the $(k,k)$ entry is also well-defined. When clear from context we omit the subscript $l$.

Next, set $$I_{k,l}:=\{A\in\M_{k-1,k-1} : \varphi_l(A)\in\Omega_{k,l}\},$$ so that $\varphi_l$ induces a bijection from $I_{k,l}$ to $\Omega_{k,l}$ (and hence $|I_{k,l}|=H_{k}(l)$). Define $$E_{k,l}:=\{A\in \M_{k,k} : A|_{k-1}=0_{(k-1)\times (k-1)}\},$$ the set of outer `edges' of non-negative integer entries and let $D_{k,l}:=\{A+B : A\in\Omega_{k,l}, B\in E_{k,l}\}$. 

We approximate the uniform distribution on $\Omega_{k,l}$ by the distribution of matrices of independent geometric random variables. To this end, let $p=\frac{k}{l+k}$, let $q=\frac{l}{l+k}$, and \begin{center}let $Y$ be a $k\times k$ matrix of independent $\mathrm{Geometric}(p,q)$ random variables,\end{center} that is $\mathbb{P}(Y_{i,j}=a)=pq^a$ for all integers $a\geq 0$. It is shown in~\cite{ditt} that for $X$ chosen uniformly from $\Omega_{k,l}$, the marginal $X_{1,1}$ approaches the geometric distribution in total variation distance as $k$ grows. We give an approximation for fixed $k$ and large $l$.

\begin{proposition}\label{dkbound}
Let $l$ and $k$ be integers at least $2$. Then each of the following holds: \begin{enumerate}\item[(i)]
conditional on $Y\in D_{k,l}$, the matrix $Y|_{k-1}$ is distributed uniformly on $I_{k,l}$;
\item[(ii)] there is some $h(k)>0$ depending only on $k$ such that $\mathbb{P}(Y\in D_{k,l})\geq h(k)$ for all $l$ sufficiently large; and
\item[(iii)]
 if $\mathcal{A}$ is any event and $X$ is a uniformly chosen $(k,l)$-intersection matrix then $\mathbb{P}(X|_{k-1}\in \mathcal{A})\leq h(k)^{-1}\mathbb{P}(Y|_{k-1}\in\mathcal{A}).$\end{enumerate}
\end{proposition}
\begin{proof}
Let $N\in I_{k,l}$. To show (i) it suffices to show that $\mathbb{P}(Y|_{k-1}=N\text{ and } Y\in D_{k,l})$ is independent of $N$. By definition of $D_{k,l}$, \begin{alignat*}{2}\mathbb{P}(Y|_{k-1}=N\text{ and } Y\in D_{k,l})&=\sum_{A=(a_{ij})\in E_{k,l}}\mathbb{P}(Y=\varphi(N)+A)
=\sum_{A\in E_{k,l}}\prod_{i,j}\mathbb{P}(Y_{i,j}=\varphi(N)_{i,j}+a_{ij})\\
&=\sum_{A\in E_{k,l}}\prod_{i,j} pq^{\varphi(N)_{i,j}+a_{ij}}
=\sum_{A\in E_{k,l}}p^{k^2}q^{\sum_{i,j}(\varphi(N)_{i,j}+a_{ij})}\\
&=\sum_{A\in E_{k,l}}p^{k^2}q^{lk+\sum_{i,j}a_{ij}}=p^{(k-1)^2}q^{lk}\sum_{A\in E_{k,l}} p^{2k-1}q^{\sum_{i,j}a_{ij}}\\
&=p^{(k-1)^2}q^{lk}\sum_{A\in E_{k,l}} \mathbb{P}(Y=A | Y\in E_{k,l})\\&=p^{(k-1)^2}q^{lk},\end{alignat*} which is independent of $N$, as desired.

Now, $$\mathbb{P}(Y\in D_{k,l})=\sum_{N\in I_{k,l}}\mathbb{P}(Y|_{k-1}=N\text{ and } Y\in D_{k,l})=\sum_{N\in I_{k,l}}p^{(k-1)^2}q^{lk}=|I_{k,l}|p^{(k-1)^2}q^{lk}.$$
From Theorem~\ref{iklsize} and the fact that $H_k(l)=|I_{k,l}|$, there is some $c(k)>0$ such that $|I_{k,l}|\geq c(k)l^{(k-1)^2}$, whence \begin{align*}\mathbb{P}(Y\in D_{k,l})&\geq c(k)l^{(k-1)^2}\left(\frac{k}{l+k}\right)^{(k-1)^2}\left(\frac{l}{l+k}\right)^{lk}\\&=c(k)k^{(k-1)^2}\left(\frac{l}{l+k}\right)^{lk+(k-1)^2}.
\end{align*} But $\left(\frac{l}{l+k}\right)^{lk+(k-1)^2}\to e^{-k^2}$ as $l\to\infty$, so 
setting $h(k)$ to be any quantity less than $c(k)k^{(k-1)^2}e^{-k^2}$ yields~(ii).

Finally, by (i) and (ii), $$\mathbb{P}(X|_{k-1}\in \mathcal{A})=\frac{\mathbb{P}(Y|_{k-1}\in\mathcal{A}\text{ and } Y\in D_{k,l})}{\mathbb{P}(Y\in D_{k,l})}\leq h(k)^{-1}\mathbb{P}(Y|_{k-1}\in\mathcal{A})$$ so (iii) holds.
\end{proof}

\begin{remark}
If we allow $k,l\to\infty$ with $l\gg k$, one can apply the asymptotic formula of Canfield and McKay~\cite[Corollary 1]{cm} to replace $h(k)$ with the explicit expression $k^{-2k}$.
\end{remark}

Proposition~\ref{dkbound} shows that events that occur with low probability as $l\to \infty$ for $Y|_{k-1}$ must also occur with low probability for the restriction to the initial $(k-1)\times (k-1)$ submatrix of random $(k,l)$-intersection matrices. We use this observation to deduce the final main result of this section.

\begin{proof}[Proof of Theorem~\ref{main4}]
By Proposition~\ref{dkbound}(iii), $$\mathbb{P}\big(|X_{1,1}-l/k|\leq f(l)\big)\leq h(k)^{-1}\mathbb{P}\big(|Y_{1,1}-l/k|\leq f(l)\big),$$
but $Y_{1,1}$ is geometric, so $$\mathbb{P}\big(|Y_{1,1}-l/k|\leq f(l)\big)=\left(\frac{l}{l+k}\right)^{l/k-f(l)}-\left(\frac{l}{l+k}\right)^{l/k+f(l)+1}.$$ Finally, $$\lim_{l\to\infty}\left(\frac{l}{l+k}\right)^{f(l)+1}=\lim_{l\to\infty}\left(\left(\frac{l}{l+k}\right)^{l}\right)^{\frac{f(l)+1}{l}}=\lim_{l\to\infty}(e^{-k})^{\frac{f(l)+1}{l}}=1$$
and similarly $\lim_{l\to\infty}\left(\frac{l}{l+k}\right)^{-f(l)}=1$, whence $\mathbb{P}\big(|Y_{1,1}-l/k|\leq f(l)\big)\to 1-1=0$ as $l\to \infty$, the result follows from the union bound.
\end{proof}
Recall the notation $\Gamma^*$ of Section~\ref{minaut} used to denote the multiset of edge multiplicities of a multigraph $\Gamma$. Since the maximum number of $(k,l)$-intersection matrices associated with each $(k,l)$-bipartite graph is independent of $l$ we immediately deduce the following corollary to Theorem~\ref{main4}.
\begin{corollary}\label{edges}
Let $k\geq 2$ and $l$ be integers and let $f(l)$ be any $o(l)$ function. Let $\Gamma$ be a uniformly chosen $(k,l)$-bipartite graph. Then asymptotically almost surely $$\min_{\omega\in\Gamma^*}|\omega-l/k|>f(l)\text{ as $l\to\infty$.}$$
\end{corollary}

\section{Synchronization}\label{sync}

This section is a coda: we show that $\S_{k\times l}$ and $\A_{k\times l}$ are non-synchronizing if
$k\geq3$ (and $l\ge2$) --- here $\A_{k\times l}$ is the alternating group acting on $(k,l)$-partitions. First we give the background in synchronization theory.

A finite-state deterministic automaton $D$ is a machine which reads a sequence of
letters from a finite alphabet $A$ and changes its internal state after each
letter depending on the letter read. The automaton $D$ is called
\emph{synchronizing} if there is some word $w$ of letters of $A$ such that reading $w$ brings the machine to a unique state, independent of its starting
state. Such a word is called a \emph{reset word}.

From another perspective, each letter of $A$ corresponds to a map
from the set $\Omega$ of states of $D$ to itself. Since these maps are composed when
a word is read, the set of possible transformations of $\Omega$ is a
transformation monoid (with a given generating set corresponding to the
letters). So we can say that a transformation monoid is \emph{synchronizing} if
it contains an element of \emph{rank}~$1$, that is, mapping the whole of 
$\Omega$ to a single point.

A permutation group of degree greater than $1$ cannot be synchronizing in this
sense, so the term was re-used for permutation groups as follows: the
permutation group $G$ on $\Omega$ is \emph{synchronizing} if and only if, for
every map $f:\Omega\to\Omega$ which is not a permutation, the monoid
$\langle G,f\rangle$ is synchronizing.

It is known that a permutation group $G$ is non-synchronizing if and only if
it is contained in the automorphism group of a non-trivial (not complete or
null) graph whose clique number is equal to its chromatic number
\cite[Corollary 4.5]{acs}.

Synchronizing permutation groups have received a lot of attention, surveyed
in~\cite{acs}. We summarise some of the properties here.

\begin{itemize}
\item A synchronizing group is primitive.
\item A synchronizing group is not contained in a wreath product with product
action.
\item A diagonal group with more than two factors in the socle is
non-synchronizing.
\end{itemize}

For the remaining O'Nan--Scott classes (affine, two-factor diagonal, and
almost simple), some groups are synchronizing and others are not. Theorem~\ref{syncthm}
is a contribution to the almost simple case.

\begin{proof}[Proof of Theorem~\ref{syncthm}]
We construct a graph $\Gamma$ with vertex set all $(k,l)$-partitions by joining $\P_1$ and $\P_2$ if and only
if they have no common part. The graph $\Gamma$ is obviously invariant under $\S_{k\times l}$. We claim that the clique number and chromatic number of $\Gamma$ are
equal. To see this, first take the colouring of $\Gamma$ as follows. Choose an element $x$ of $[kl]$. For each $(l-1)$-subset $A\subseteq [kl]\setminus\{x\}$,
assign colour $c_A$ to a partition $\P$ if the part of $\P$ containing $x$ is
$\{x\}\cup A$. Each colour class is an independent set, so this is a
proper colouring with ${kl-1\choose l-1}$ colours.

To find a clique with size equal to the number of colours, we use Baranyai’s celebrated theorem~\cite{baranyai}: the $l$-sets in $[kl]$ can be partitioned into classes, each of which is a partition of $[kl]$. Of the resulting partitions, no two share a part, and so they form a clique of size ${kl-1\choose l-1}$ in the graph. 
\end{proof}

This construction fails for $k=2$ since, in this case, the graph is complete:
if two partitions share a part, they are equal. In fact, in this case the
group is $2$-transitive (and hence synchronizing) for $l=2$ and $l=3$, and
non-synchronizing for $l=4$ and $l=6$. It is synchronizing for $l=5$, but the
proof is computational.


\section*{Acknowledgments} 
We are grateful to Brendan McKay for his helpful suggestions, and to the anonymous referees for their comments which have improved the quality of the paper.



\begin{dajauthors}
\begin{authorinfo}[Cam]
  Peter J. Cameron\\
  School of Mathematics and Statistics\\
  University of St Andrews\\
  St Andrews, UK\\
  pjc20\imageat{}st-andrews\imagedot{}ac\imagedot{}uk
\end{authorinfo}
\begin{authorinfo}[del]
  Coen del Valle\\
  School of Mathematics and Statistics\\
  The Open University\\
  Milton Keynes, UK\\
  coen.del-valle\imageat{}open\imagedot{}ac\imagedot{}uk
\end{authorinfo}
\begin{authorinfo}[Ron]
  Colva M. Roney-Dougal\\
  School of Mathematics and Statistics\\
  University of St Andrews\\
  St Andrews, UK\\
  colva.roney-dougal\imageat{}st-andrews\imagedot{}ac\imagedot{}uk
\end{authorinfo}
\end{dajauthors}


\begin{thebibliography}{99}

\bibitem{acs}
J. Ara\'ujo, P.J. Cameron and B. Steinberg,
\emph{Between primitive and 2-transitive: Synchronization and its friends}, EMS Surv.
Math. Sci. \textbf{4} (2017), 101--184.

\bibitem{babai}
L. Babai, \emph{Almost all Steiner triple systems are asymmetric}, Ann. Discrete Math. \textbf{7} (1980), 37--39.

\bibitem{baranyai}
Z. Baranyai, \emph{On the factorization of the complete uniform hypergraph}, Colloq. Math. Soc. Janos Bolyai \textbf{10} (1975), 91--108.

\bibitem{cm}
R.E. Canfield, B.D. McKay, \emph{Asymptotic enumeration of integer matrices with large equal row and column sums}, Combinatorica \textbf{30(6)} (2010), 655--680.

\bibitem{cds}
S. Chatterjee, P. Diaconis, A. Sly, \emph{Properties of uniform doubly stochastic matrices}, arXiv preprint: \url{https://arxiv.org/abs/1010.6136}, (2010).

\bibitem{ditt}
S. Dittmer, H. Lyu, I. Pak, \emph{Phase transition in random contingency tables with non-uniform margins}, Trans. Amer. Math. Soc. \textbf{373(12)} (2020), 8313--8338.

\bibitem{erdos}
P. Erd\H{o}s, \emph{On an elementary proof of some asymptotic formulas in the theory of partitions}, Ann. of Math. (2) \textbf{43} (1942), 437--450.

\bibitem{erren}
P. Erd\H{o}s, A. R\'{e}nyi, \emph{Asymmetric graphs}, Acta Math. Acad. Sci. Hungar. \textbf{14} (1963), 295--315.

\bibitem{frucht}
R. Frucht, \emph{Herstellung von {G}raphen mit vorgegebener abstrakter {G}ruppe}, Compositio Math. \textbf{6} (1939), 239--250.

\bibitem{james2}
J.P. James, \emph{Arbitrary groups as two-point stabilisers of symmetric groups acting on partitions}, J. Algebraic Combin. \textbf{24(4)} (2006), 355--360.

\bibitem{james}
J.P. James, \emph{Partition actions of symmetric groups and regular bipartite graphs}, Bull. London Math. Soc. \textbf{38(2)} (2006), 224--232.

\bibitem{jordan}
C. Jordan, \emph{Sur les assemblages de lignes} J. Reine Angew. Math. \textbf{70} (1869), 185--190.

\bibitem{kantor}
W.M. Kantor, \emph{Universal families of permutation groups}, J. Algebraic Combin. \textbf{28(3)} (2008), 351--363.

\bibitem{mw}
B. D. McKay, I.M. Wanless, \emph{On the number of Latin squares}, Ann. Comb. \textbf{9} (2005), 335--344.

\bibitem{men}
E. Mendelsohn, \emph{On the groups of automorphisms of Steiner triple and quadruple systems}, J. Combin. Theory Ser. A \textbf{25(2)} (1978), 97--104.

\bibitem{phelps}
K.T. Phelps, \emph{Latin square graphs and their automorphism groups}, Ars Combin. \textbf{7} (1979), 273--299.

\bibitem{sab}
G. Sabidussi, \emph{Graphs with given group and given graph-theoretical properties}, Canadian J. Math.  \textbf{9} (1957), 515--525.

\bibitem{ser}
\'A. Seress, \emph{Permutation {G}roup {A}lgorithms}, Cambridge Tracts in Mathematics \textbf{152}, Cambridge University Press.

\bibitem{stanley}
R.P. Stanley, \emph{Linear homogeneous {D}iophantine equations and magic labelings of graphs}, Duke Math. J. \textbf{40} (1973), 607--632.

\end{thebibliography}
\end{document}